\documentclass[a4paper,reqno,10pt]{amsart}

\usepackage{psfrag}

\usepackage{setspace}
\usepackage[totalwidth=14cm,totalheight=22cm]{geometry}
\usepackage[T1]{fontenc}
\usepackage[dvips]{graphicx}
\usepackage{epsfig}
\usepackage{amsmath,amssymb,amscd,amsthm}
\usepackage{amsfonts}
\usepackage{amssymb}
\usepackage{subfigure}
\usepackage[latin1]{inputenc}
\usepackage{latexsym}
\usepackage{graphics}
\usepackage{colortbl}
\usepackage{color}
\usepackage{ae}
\usepackage{enumitem}
\usepackage[all]{xy}
\usepackage{cancel}
\usepackage{bbm}

\usepackage{relsize}

\newcommand{\red}[1]{}

\makeatletter
\newtheorem*{rep@theorem}{\rep@title}
\newcommand{\newreptheorem}[2]{%
\newenvironment{rep#1}[1]{%
 \def\rep@title{#2 \ref{##1}}%
 \begin{rep@theorem}}%
 {\end{rep@theorem}}}
\makeatother

\makeatletter
\newtheorem*{rep@cor}{\rep@title}
\newcommand{\newrepcor}[2]{%
\newenvironment{rep#1}[1]{%
 \def\rep@title{#2 \ref{##1}}%
 \begin{rep@cor}}%
 {\end{rep@cor}}}
\makeatother

\makeatletter
\newtheorem*{rep@prop}{\rep@title}
\newcommand{\newrepprop}[2]{%
\newenvironment{rep#1}[1]{%
 \def\rep@title{#2 \ref{##1}}%
 \begin{rep@prop}}%
 {\end{rep@prop}}}
\makeatother

\newtheorem{cor}{Corollary}[section]

\newtheorem{prop}[cor]{Proposition}

\newrepcor{cor}{Corollary}
\newreptheorem{theorem}{Theorem}
\newrepprop{prop}{Proposition}

\newtheorem*{theorem*}{Theorem}
\newtheorem{lemma}[cor]{Lemma}

\theoremstyle{definition}
\newtheorem{defi}[cor]{Definition}
\theoremstyle{remark}
\newtheorem{remark}[cor]{Remark}
\newtheorem*{remark*}{Remark}
\newtheorem{example}[cor]{Example}

\newcommand{\weil}{{\mathrm{W}\!\mathrm{P}}}
\newcommand{\rham}{{\mathrm{d}\!\mathrm{R}}}

\newcommand{\R}{{\mathbb R}}

\newcommand{\dev}{\mbox{dev}}

\newcommand{\Hyp}{\mathbb{H}}

\newcommand{\Isom}{\mathrm{Isom}}

\newcommand{\Ad}{{\mathrm{Ad}}}

\newcommand{\Diffeo}{\mbox{Diff}}

\newcommand{\tr}{\mbox{\rm tr}}

\newcommand{\isom}{\mathfrak{isom}}

\newcommand{\SO}{\mathrm{SO}}
\newcommand{\so}{\mathfrak{so}}

\newcommand{\E}{\mathcal{E}}
\newcommand{\V}{\mathcal{V}}

\newcommand{\BG}{\mathcal{BG}}
\newcommand{\Hol}{\mathrm{Hol}}

\newcommand{\Teich}{\mathcal{T}}

\DeclareFontFamily{OT1}{pzc}{}
\DeclareFontShape{OT1}{pzc}{m}{it}{<-> s * [0.99] pzcmi7t}{}
\DeclareMathAlphabet{\mathscr}{OT1}{pzc}{m}{it}

\begin{document}

\setcounter{secnumdepth}{3}
\setcounter{tocdepth}{2}

\title[Generalization of a formula of Wolpert for balanced geodesic graphs]{Generalization of a formula of Wolpert for balanced geodesic graphs on closed hyperbolic surfaces}

\author[Fran\c{c}ois Fillastre]{Fran\c{c}ois Fillastre}
\address{Fran\c{c}ois Fillastre: Universit\'e de Cergy-Pontoise
UMR CNRS 8088
Cergy-Pontoise, 95000, France.} \email{francois.fillastre@u-cergy.fr}

\author[Andrea Seppi]{Andrea Seppi}
\address{Andrea Seppi: CNRS and Universit\'e Grenoble Alpes, 100 Rue des Math\'ematiques, 38610 Gi\`eres, France.} \email{andrea.seppi@univ-grenoble-alpes.fr}



\begin{abstract}
A well-known theorem of Wolpert shows that the Weil--Petersson symplectic form on Teichm\"uller space, computed on two infinitesimal twists along simple closed geodesics on a fixed hyperbolic surface, equals the sum of the cosines of the intersection angles. We define an infinitesimal deformation starting from a more general object, namely a balanced geodesic graph, by which any tangent vector to Teichm\"uller space can be represented. We then prove a generalization of Wolpert's formula for these deformations. In the case of simple closed curves, we recover the theorem of Wolpert.
\end{abstract}

\maketitle


\section{Introduction}
Given a closed oriented surface $S$ of genus $g\geq 2$, Teichm\"uller space $\Teich(S)$ is classically defined as the space of complex structures on $S$, up to biholomorphisms isotopic to the identity. By the uniformization theorem, it is naturally identified to what is sometimes called 
Fricke space $\mathcal F(S)$, namely the space of hyperbolic metrics on $S$ up to isometries
isotopic to the identity.

The Weil--Petersson metric is a K\"ahler metric on $\Teich(S)$, whose definition only involves the point of view of $\Teich(S)$ as the space of complex structures on $S$. However, it turns out to be an extremely interesting object from the hyperbolic geometric point of view. An example of the bridge between the two viewpoints is in fact provided by Wolpert's theorem (\cite{wolpertelementary,wolpertformula}), which we now briefly recall. Given two simple closed geodesics $c,c'$ on a fixed closed hyperbolic surface $(S,h)$, let us denote by $t_c$ and $t_c'$ the \emph{infinitesimal twists} along $c$ and $c'$, namely 
$$t_c=\left.\frac{d}{dw}\right|_{w=0}E_l^{w\cdot c}(h)~,$$
where $E_l^{w\cdot c}(h)$ is the new hyperbolic surface obtained from $h$ by a \emph{left earthquake} 
along $c$ --- that is, by cutting along $c$ and re-glueing after a left twist of length $w$. Then Wolpert showed that, if $\omega_\weil$ denotes the symplectic form of the Weil--Petersson K\"ahler metric, then
\begin{equation} \label{eq wolpert formula}
\omega_\weil(t_{c},t_{c'})=\frac{1}{2}\sum_{p\in c\cap c'}\cos \theta_p~,
\end{equation}
where, for every point of intersection $p$ between $c$ and $c'$, $\theta_p$ denotes the angle of intersection at $p$.

\subsection*{Balanced geodesic graphs}

In this paper, we will define a more general type of infinitesimal deformations on $\Teich(S)$, which generalize twists along simple closed geodesics. These will be associated to a \emph{balanced geodesic graph}, namely a weighted graph $(\mathcal G,\mathbf w)$ on $(S,h)$ whose edges 
are geodesic segments and whose weights satisfy a \emph{balance condition} at every vertex $p$, namely:
\begin{equation}\label{eq balance intro} {\sum_{p\in e} w_{e} v_{e}}=0~, \end{equation}
where the sum is over all edges $e$ which are adjacent to the vertex $p$, $w_e$ denotes the weight of $e$ and $v_e$ is the tangent vector to $e$ in $T_p S$. 

Clearly, every simple closed geodesic can be regarded as a balanced geodesic graph, just by adding a number of vertices so as to make sure that every edge is a segment, and choosing the same weight $w$ for all edges. The balance condition is then trivially satisfied.
In fact, the infinitesimal deformation $t_{(\mathcal G,\mathbf w)}$ we define reduces to the infinitesimal twist in this case.

In order to define $t_{(\mathcal G,\mathbf w)}$, it is easier to use the identification of $\Teich(S)$ to the space of discrete and faithful representations of $\pi_1(S)$ into $\Isom(\Hyp^2)$ (i.e. the group of orientation-preserving isometries of the hyperbolic plane), up to conjugacy. The identification is obtained by associating to a hyperbolic metric $h$ on $S$ its \emph{holonomy representation} $\rho$. The tangent space to the space of representations up to conjugacy is then known to be identified (see \cite{Goldman}) to the group cohomology $H^1_{\Ad\rho}(\pi_1(S),\isom(\Hyp^2))$. The tangent vector $t_{(\mathcal G,\mathbf w)}$ is then defined by lifting to the universal cover of $S$ a generic closed loop representing $\gamma\in\pi_1(S)$, and taking the weighted sum of the infinitesimal twists along the lifts of the edges of $\mathcal G$ met by the lift of the closed loop. The balance condition ensures that  $t_{(\mathcal G,\mathbf w)}$ is well-defined. If $(\mathcal G,\mathbf w)=(c,1)$ is a simple closed geodesic with weight 1, $t_{(\mathcal G,\mathbf w)}$ actually coincides with the variation of the holonomy of the twisted metrics $E_l^{w\cdot c}(h)$, and we thus recover the infinitesimal twist $t_c$.

\subsection*{Motivations}

 \red{Besides the main result explained below, the main motivation for this paper is the introduction of the balanced geodesic graphs. Weighted multicurves are a particular case of our balanced geodesic graph, as well as a particular case of measured geodesic laminations. While the original motivation behind measured geodesic laminations is to construct a natural completion of the space of weighted multicurves, the space of balanced geodesic graphs seems an interesting object to consider because it is endowed with a natural vector space structure (described below), thus including the ``span'' of weighted multicurves. In turn, balanced geodesic graphs have a more combinatorial nature, easier to handle with than measured geodesic laminations ---see e.g. \cite{SB} where Wolpert formula where generalized to the case of geodesic laminations.}

\red{Let us highlight} two main motivations behind the study of balanced geodesic graphs. The first motivation is directly related to hyperbolic geometry, in the spirit of the study of earthquakes of hyperbolic surfaces, see for instance \cite{thurstonearth,bonahonearth,mcmullerearth}. In fact, in \cite{flippable} \emph{left/right flippable tilings} were introduced. These are (non-continuous) transformations between hyperbolic surfaces, associated to certain geodesic graphs on $(S,h)$ such that the faces of the graph are divided into \emph{black} and \emph{white} faces, and the transformation is obtained by \emph{flipping} the black faces, \red{see Figure~\ref{fig:flip}}. A balanced geodesic graph $(\mathcal G,\mathbf w)$ can be interpreted as a tangent vector to a path of flippable tilings on hyperbolic surfaces, where the black faces are collapsing to the vertices of $\mathcal G$ and the weights $\mathbf w$ are the derivatives of the lengths of such black faces. \red{Flippable tilings} are thus a generalization of the deformations by earthquakes along simple closed curves, which have an infinitesimal twist as tangent vector.  \red{Balanced geodesic graph, that are the object of the present paper, are generalizations of (left, say) infinitesimal earthquake}.

The second motivation comes from the deep relation between Teichm\"uller theory and geometric structures on three-dimensional manifolds. In the case of hyperbolic structures in dimension three, this is an important phenomenon, which goes back to Bers' simultaneous uniformization theorem for quasi-Fuchsian manifolds \cite{bers}, and has been widely developed \cite{BrockWPconvexcore,bonahonbouts,uhlenbeck,taubes,seppiminimal}. Analogous of quasi-Fuchsian manifolds in Lorentzian geometry are \emph{maximal globally hyperbolic manifolds}, and their relation with Teichm\"uller theory has been initiated in \cite{mess}. Here we are mostly interested in \emph{flat} maximal globally hyperbolic manifolds, as studied in \cite{Bonsante,barbot,bonseppicodazzi}. The relevant point here is that a balanced geodesic graph {with positive weights, and such that $S\setminus\mathcal{G}$ is a disjoint union of convex polygons,} is naturally the \emph{dual} to convex polyhedral surface in a flat globally hyperbolic manifold $M$ of dimension three {homeomorphic to $S\times \R$}. Moreover, as observed by Mess, the isomorphism between Minkowski space $\R^{2,1}$ and the Lie algebra $\isom(\Hyp^2)\cong\so(2,1)$ induces a correspondence between the tangent space of $\Teich(S)$, in the model $H^1_{\Ad\rho}(\pi_1(S),\isom(\Hyp^2))$ of the representation variety, and the translation part of the holonomies of manifolds $M$ as above.

This enables us to show that the map which associated to a balanced geodesic graph $(\mathcal G,\mathbf w)$ the deformation $t_{(\mathcal G,\mathbf w)}$ in $T_{[h]}\Teich(S)$ is surjective. In other words, \emph{any} tangent vector to $\Teich(S)$ can be represented as the deformation associated to some balanced geodesic graph. This is not true if one only considers (weighted) simple closed geodesics. Hence in this sense our main result below, which extends Wolpert's formula, is quite more general since it can be used to represent the Weil--Petersson form
applied to \emph{any} two tangent vectors to $\Teich(S)$.

\subsection*{Main result}

Let us now come to the statement of the main result. Recall that $\omega_\weil$ denotes the symplectic form of the Weil--Petersson metric on $\Teich(S)$.

\begin{theorem*}
Let $(\mathcal G,\mathbf w)$ and $(\mathcal G',\mathbf w')$ be two balanced geodesic graphs on the hyperbolic surface $(S,h)$. Then
\begin{equation} \label{eq main formula}
\omega_\weil(t_{(\mathcal G,\mathbf w)},t_{(\mathcal G',\mathbf w')})=\frac{1}{2}\sum_{p\in e\cap e'}w_ew_e'\cos \theta_{e,e'}~,
\end{equation}
where $e$ and $e'$ are intersecting edges of $\mathcal G$ and $\mathcal G'$ and $\theta_{e,e'}$ is the angle of intersection
between $e$ and $e'$ according to the orientation of $S$.
\end{theorem*}

Therefore, when $(\mathcal G,\mathbf w)=(c,1)$ is a simple closed geodesic with weight 1, we  recover Wolpert's formula \eqref{eq wolpert formula}. Let us remark again that our construction permits to represent \emph{any} tangent vector in $T_{[h]}\Teich(S)$ as $t_{(\mathcal G,\mathbf w)}$ for some balanced geodesic graph $(\mathcal G,\mathbf w)$, and thus our result is in a greater generality than Wolpert's. 

There is a small caveat in interpreting \eqref{eq main formula}. If the two geodesic graphs  $\mathcal G$ and $\mathcal G'$ do not have vertices in common and intersect transversely, it is clear what the intersection points and angles are. On the other hand, if $\mathcal G$ and $\mathcal G'$ have some non-transverse intersection or share some vertex, the points of intersection must be counted by perturbing one of the two graphs by a small isotopy, so as to make the intersection transverse. See Figures \ref{fig:perturbation1} and \ref{fig:perturbation2}. This is just a caveat for counting points of intersections (it turns out that the result not depend on the chosen perturbation, as a consequence of the balance condition \eqref{eq balance intro}), while of course the angles $\cos\theta_{e,e'}$ are the angles between the original geodesic edges $e,e'$.

To conclude the introduction, let us mention that the proof of our main result relies on two main tools. The first tool is a theorem of Goldman (\cite{Goldman}) which shows that the Weil--Petersson symplectic form on $\Teich(S)$ equals (under the holonomy map) a form defined only in terms of the group cohomology $H^1_{\Ad\rho}(\pi_1(S),\isom(\Hyp^2))$. The second tool is de Rham cohomology with values in certain flat vector bundles over $S$ of rank $3$. In this setup, our proof becomes rather elementary, and we thus also provide a simple proof of Wolpert's formula for simple closed geodesics as a particular case.


\subsection*{Acknowledgements}
We would like to thank Jean-Marc Schlenker for many discussions and encouragements, Brice Loustau for his interest in our work, and Scott Wolpert for reading our paper and suggesting some improvements in Section~\ref{sec balanced graph}. We are grateful to an anonymous referee whose comments helped to improve the presentation. The second author would like to thank Gabriele Mondello for several discussions and for the help in figuring out the correct factors from the literature. 

\section{Teichm\"uller space and Weil--Petersson metric}

Let $S$ be a closed oriented surface of genus $g\geq 2$. The \emph{Teichm\"uller space} of $S$ is defined as:
$$\Teich (S)=\{\text{complex structures on }S\}/\Diffeo_0(S)~,$$
where $\Diffeo_0(S)$ is the group of diffeomorphisms of $S$ isotopic to the identity, and it acts by pre-composition of a complex atlas. Namely, two complex structures on $S$ are equivalent in 
$\Teich (S)$ if and only if there exists a biholomorphism isotopic to the identity. In this section we collect some preliminary results on Teichm\"uller space, spaces of representations of the fundamental group of $S$, and Weil--Petersson metric.

\subsection*{Weil--Petersson metric}

Teichm\"uller space $\Teich(S)$ is a manifold of real dimension $3|\chi(S)|$, and is endowed with a structure of complex manifold. Moreover, it is endowed with several metric structures, one of which is the \emph{Weil--Petersson metric}, which turns out to be a natural K\"ahler structure on $\Teich(S)$. Let us recall briefly its definition.

Let us fix a complex structure $X$ on $S$. It is known that the tangent space $T_{[X]}\Teich(S)$  of $\Teich(S)$ is naturally identified to a quotient of the vector space $BD(X)$ of Beltrami differentials, namely sections of the vector bundle $\overline K_X\otimes K_X^{-1}$, where $K_X$ is the canonical bundle of $(S,X)$. More precisely, $T_{[X]}\Teich(S)\cong BD(X)/BD_0(X)$, where $BD_0(X)$ is the subspace of Beltrami differentials which induce trivial infinitesimal deformations of $X$.

On the other hand, the cotangent space $T^*_{[X]}\Teich(S)$ is naturally identified to the space of holomorphic quadratic differentials $QD(X)=H^0(S,K_X^2)$. The identification is given by the pairing 
on $QD(X)\times BD(X)$
defined by
$$(\phi,\mu)\mapsto\int_S \phi\mu~.$$
In fact, if in local complex coordinates $\phi=\phi(z)dz^2$ and $\mu=\mu(z)d\bar z/dz$, then $\phi\mu=\phi(z)\mu(z)dz\wedge d\bar z$ is a quantity which can be naturally integrated over $S$.
The fundamental property is then the fact that for every $\phi\in QD(X)$, 
$$\int_S \phi\mu=0\quad\textrm{for every }\mu\in BD(X)\Leftrightarrow \mu\in BD_0(X)~.$$
Hence we have a vector space isomorphism $QD(X)\cong (BD(X)/BD_0(X))^*$.

The Weil--Petersson product is then easily defined on the cotangent space, by
$$(\phi,\phi')\mapsto \int_S \frac{\phi\overline{\phi'}}{h_X}~,$$
where $h_X$ is the unique hyperbolic metric (i.e. Riemannian metric of constant curvature $-1$) compatible with the complex structure $X$, provided by the Uniformization Theorem (\cite{koebeunif}). By a similar argument as above, the quantity $\phi\overline{\phi'}/h_X$ is indeed of the correct type to be integrated on $S$.

\subsection*{Hyperbolic metrics and Fuchsian representations}

We will be using other two important models of $\Teich(S)$. 

\begin{enumerate}
\item By the aforementioned Uniformization Theorem, given any complex structure $X$ on $S$, there is a unique hyperbolic metric $h_X$ on $S$ compatible with $X$. Let us define the \emph{Fricke space} of $S$ as:
$$\mathcal F(S)=\{\text{hyperbolic metric on }S\}/\Diffeo_0(S)~,$$
where $\Diffeo_0(S)$ clearly acts by pull-back. It is easy to check that the map $X\to h_X$ is equivariant for the actions of $\Diffeo_0(S)$, and therefore induces a diffeomorphism
\begin{equation}
U:\Teich(S)\to\mathcal F(S)~.
\end{equation}
The inverse of $U$ is simply the map which associates to a hyperbolic metric $h$ the complex structure induced by $h$, which is obtained by choosing local isothermal coordinates.
\item Given a hyperbolic metric $h$ on $S$, let $\pi:\widetilde S\to S$ be the universal cover of $S$. Then $\pi^*h$ is a hyperbolic metric, which is complete since $S$ is compact, on the simply connected surface $S$. Hence $(\widetilde S,\pi^*h)$ is isometric to  the hyperbolic plane $\Hyp^2$ by \cite[Theorem 8.6.2]{opac-b1124711}. Such an isometry (chosen to be orientation-preserving) is unique up to post-compositions with elements in group $\Isom(\Hyp^2)$ of orientation-preserving isometries of $\Hyp^2$, and is called \emph{developing map}. Let us denote it by $\dev:(\widetilde S,\pi^*h)\to\Hyp^2$. It turns out to be equivariant for some representation $\rho:\pi_1(S)\to\Isom(\Hyp^2)$, called the \emph{holonomy map}. That is:
$$\dev\circ\gamma=\rho(\gamma)\circ\dev$$
for every $\gamma\in\pi_1(S)$.
Since $\dev$ is well-defined up to post-composition, $\rho$ is well-defined up to conjugacy by elements of $\Isom(\Hyp^2)$. This provides a map 
$$\mathcal F(S)\to \chi(\pi_1(S),\Isom(\Hyp^2))~,$$
where $\chi(\pi_1(S),\Isom(\Hyp^2))$ is the \emph{character variety}
$$\mathrm{Hom}(\pi_1(S),\Isom(\Hyp^2))/\!\!/\Isom(\Hyp^2)~.$$
By a theorem of Goldman \cite{goldmanthesis}, this map is a diffeomorphism onto the space of faithful and discrete representations (called \emph{Fuchsian representations}) up to conjugacy, which is precisely a connected component of $\chi(\pi_1(S),\Isom(\Hyp^2))$:
\begin{equation}
\Hol:\mathcal F(S)\to \chi^{f\!d}(\pi_1(S),\Isom(\Hyp^2))~,
\end{equation}
where 
$$\chi^{f\!d}(\pi_1(S),\Isom(\Hyp^2))=\mathrm{Hom}^{f\!d}(\pi_1(S),\Isom(\Hyp^2))/\Isom(\Hyp^2)$$ 
and $\mathrm{Hom}^{f\!d}$ denotes the faithful and discrete representations.

\end{enumerate}

\subsection*{Group cohomology}

The model of Teichm\"uller space $\Teich(S)$ as $\chi^{f\!d}(\pi_1(S),\Isom(\Hyp^2))$ enables to give a simple description of the tangent space. In fact, using the differentials of the maps $U$ and $\Hol$, we can identify
$$T_{[X]}\Teich(S)\cong T_{[h_X]}\mathcal F(S)\cong T_{[\rho]}\chi(\pi_1(S),\Isom(\Hyp^2))~,$$
where $[\rho]=\Hol([h_X])$.  
In \cite{Goldman}, the tangent space to the space of representations is described as the group cohomology with values in the Lie algebra $\isom(\Hyp^2)$:
$$T_{[\rho]}\chi(\pi_1(S),\Isom(\Hyp^2))\cong H^1_{\Ad\rho}(\pi_1(S),\isom(\Hyp^2))~.$$
The vector space $H^1_{\Ad\rho}(\pi_1(S),\isom(\Hyp^2))$ is the quotient
$$H^1_{\Ad\rho}(\pi_1(S),\isom(\Hyp^2))=\frac{Z^1_{\Ad\rho}(\pi_1(S),\isom(\Hyp^2))}{B^1_{\Ad\rho}(\pi_1(S),\isom(\Hyp^2))}~,$$
where:
\begin{itemize}
\item $Z^1_{\Ad\rho}(\pi_1(S),\isom(\Hyp^2))$ is the space of \emph{cocycles} $\tau:\pi_1(S)\to\isom(\Hyp^2)$ with respect to the adjoint action of $\rho$, that is, functions with values in the Lie algebra $\isom(\Hyp^2)$
satisfying:
\begin{equation} \label{eq cocycle condition}
\tau(\gamma\eta)=\Ad\rho(\gamma)\cdot\tau(\eta)+\tau(\gamma)~.
\end{equation}
This is essentially the condition of being a representation of $\pi_1(S)$ into $\Isom(\Hyp^2)$, at first order.
\item $B^1_{\Ad\rho}(\pi_1(S),\isom(\Hyp^2))$ is the space of \emph{coboundaries}, namely cocycles of the form
\begin{equation} \label{eq defi coboundary}
\tau(\gamma)=\Ad\rho(\gamma)\cdot\tau_0-\tau_0~,
\end{equation}
for some $\tau_0\in\isom(\Hyp^2)$.
This is the first-order condition for a deformation of being trivial in $\chi(\pi_1(S),\Isom(\Hyp^2))$, that is, of being tangent to a path of representations $\rho_t$ obtained from $\rho$ by conjugation.
\end{itemize}

\subsection*{Goldman symplectic form}

In the fundamental paper \cite{Goldman}, Goldman introduced a symplectic form $\omega_{\mathrm G}$ on the space $\chi(\pi_1(S),\isom(\Hyp^2))$ (actually, the construction holds when replacing $\isom(\Hyp^2)$ by a more general Lie group $G$) and showed that it coincides  (up to a factor) with the symplectic form $\omega_{\weil}$ of the Weil--Petersson Hermitian metric.

The Goldman form is defined as follows. Recall from the previous section that the tangent space $T_{[\rho]}\chi(\pi_1(S),\Isom(\Hyp^2))$ is identified to the group cohomology $H^1_{\Ad\rho}(\pi_1(S),\isom(\Hyp^2))$. Then one can define a pairing
$$H^1_{\Ad\rho}(\pi_1(S),\isom(\Hyp^2))\times H^1_{\Ad\rho}(\pi_1(S),\isom(\Hyp^2))\xrightarrow{\omega_{\mathrm G}} H^2_{\Ad\rho}(\pi_1(S),\R)\cong\R~.$$
The first arrow is obtained by the cup product in group cohomology, paired by using the Killing form of $\isom(\Hyp^2)$. The identification between $H^2(\pi_1(S),\R)$ and $\R$ is then given by evaluation on the fundamental top-dimensional class of the closed oriented manifold $S$. 

Recall that the map $\Hol\circ U$ associates to the Teichm\"uller class of a complex structure $X$ on $S$ the holonomy representation of the uniformizing hyperbolic metric $h_X$. The differential of $\Hol\circ U$ should hence be consider as a vector space isomorphism
$$d(\Hol\circ U):T_{[X]}\Teich(S)\to H^1_{\Ad\rho}(\pi_1(S),\isom(\Hyp^2))~.$$ 
In \cite{Goldman}, Goldman proved:
\begin{equation} \label{eq goldman theorem}
(\Hol\circ U)^*\omega_{\mathrm G}=\frac{1}{4}\omega_\weil~.
\end{equation}
Concerning Equation \eqref{eq goldman theorem}, we remark that in Goldman's original paper \cite{Goldman} there is a different factor appearing. This is due to two reasons:
\begin{itemize}
\item Goldman uses the trace form in the $\mathfrak{sl}(2,\R)$ model of $\isom(\Hyp^2)$, which is defined as $\mathfrak b(X,Y)=\tr(XY)$, instead of the Killing form which turns out to be $\kappa=4\mathfrak b$.
\item The original theorem of Goldman concerns the Weil--Petersson metric on $T^*\Teich(S)$, hence the 
choice of an identification between $T_{[X]}\Teich(S)$ and $T^*_{[X]}\Teich(S)$ may result in different coefficients for the Weil--Petersson metric.
\end{itemize}
As a reference for Equation \eqref{eq goldman theorem}, which is the formula we actually need in this paper, see \cite[Section 2]{bricegt}.

There actually is another description of the Goldman form in terms of de Rham cohomology, which will be introduced in Section \ref{sec de rham}.

\section{Some properties of the hyperboloid model} \label{sec hyperboloid}

It will be useful for this paper to consider the hyperboloid model of $\Hyp^2$. Namely, let us consider (2+1)-dimensional Minkowski space, which is the vector space $\R^3$ endowed with the standard bilinear form of signature $(2,1)$:
$$\R^{2,1}=(\R^3,\langle x,x'\rangle=x_1x_1'+x_2x_2'-x_3x_3')~.$$
It turns out that the induced bilinear form on the upper connected component of the two-sheeted hyperboloid (which is simply connected) gives a complete hyperbolic metric. It is thus isometric to $\Hyp^2$, again by  \cite[Theorem 8.6.2]{opac-b1124711}. Hence we will identify
$$\Hyp^2=\{x\in\R^{2,1}:\langle x,x\rangle=-1\,,\,x_3>0\}~.$$

\subsection*{Description of the Lie algebra}

By means of this identification, we have
$$\Isom(\Hyp^2)=\SO_0(2,1)~,$$
namely, the group of orientation-preserving isometries of $\Hyp^2$ is the identity component in the group of linear isometries of the Minkowski bilinear form. We then also have the following identification for the Lie algebra:
$$\isom(\Hyp^2)=\so(2,1)~,$$
where $\so(2,1)$ are skew-symmetric matrices with respect to the Minkowski metric.
A useful description for this Lie algebra is provided by the Minkowski cross product, which is the analogue for $\R^{2,1}$ of the classical Euclidean cross product. This provides an isomorphism
$$\Lambda:\R^{2,1}\to \so(2,1)~,$$
namely
$$\Lambda(x)(y)=y\boxtimes x\in\R^{2,1}$$
for any $y\in\R^{2,1}$. More explicitly,
$$\Lambda\begin{pmatrix} x_1 \\ x_2 \\ x_3 \end{pmatrix}=\begin{pmatrix} 0 & x_3 & -x_2 \\ -x_3 & 0 & x_1 \\ -x_2 & x_1 & 0 \end{pmatrix}\,.$$

\subsection*{Hyperbolic isometries}

An example is given by hyperbolic isometries. Every geodesic $\ell$ of $\Hyp^2$ is of the form $\ell=\Hyp^2\cap x^\perp\subset \R^{2,1}$, for some $x\in\R^{2,1}$ with $\langle x,x\rangle=1$, {and $x^ \perp$ is the plane orthogonal to the vector $x$ for $\langle\cdot,\cdot\rangle$}. Moreover, the orientation of $\Hyp^2$ and the direction of $x$ determine an orientation of $\ell$. Then we define $T_t(\ell)$ the hyperbolic isometry which preserves $\ell$ setwise and translates every point of $\ell$ by a length $t$ according to the orientation of $\ell$. We will also denote 
$$\mathfrak t(\ell):=\left.\frac{d}{dt}\right|_{t=0}T_t(\ell)~,$$
namely, $\mathfrak t(\ell)$ is the generator of the 1-parameter group $T_t(\ell)$, or in other words the infinitesimal translation along $\ell$.

\begin{example}
If we pick $x=(1,0,0)$, then $\ell=\Hyp^2\cap\{x_1=0\}$ (see also Figure \ref{fig:hyperboloid}) and 
$$T_t(\ell)=\begin{pmatrix} 1 & 0 & 0 \\ 0 & \cosh t & \sinh t \\ 0 & \sinh t & \cosh t  \end{pmatrix}\,.$$
Hence the infinitesimal isometry is
$$\mathfrak t(\ell):=\left.\frac{d}{dt}\right|_{t=0}T_t(\ell)=\begin{pmatrix} 0 & 0 & 0 \\ 0 & 0 & 1 \\ 0 & 1 & 0  \end{pmatrix}=\Lambda\begin{pmatrix} 1 \\ 0 \\ 0 \end{pmatrix}~.$$
\end{example}


\begin{figure}
\begin{center}
\psfrag{v}{$x$}
\psfrag{d}{$\ell$}
\psfrag{h}{$\mathbb{H}^2$}
\includegraphics[scale=1]{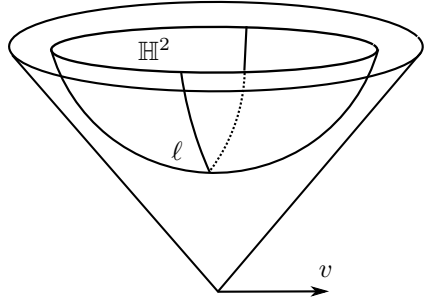}
\end{center}\caption{In the hyperboloid model $\Hyp^2$, the geodesic $\ell=\Hyp^2\cap x^\perp$.
\label{fig:hyperboloid}}
\end{figure}

In general, if $\ell$ is an (oriented) geodesic, intersection of $\Hyp^2$ and $x^\perp$ endowed with the induced orientation, then the following formula holds:
\begin{equation}
\mathfrak t(\ell)=\Lambda(x)~.
\end{equation}

\subsection*{Additional structures on $\so(2,1)$}

Finally, two important properties of the isomorphism $\Lambda$ are the following. See also \cite[Section 2]{ksurfaces}
\begin{itemize}
\item $\Lambda$ is equivariant for the actions of $\SO_0(2,1)$: the standard action on $\R^{2,1}$ and the adjoint action on $\so(2,1)$. Namely, for every $\eta\in\SO_0(2,1)$,
\begin{equation}
\Lambda(\eta\cdot x)=\Ad \eta\cdot \Lambda(x)~.
\end{equation}
\item $\Lambda$ is an isometry between the Minkowski metric and the Killing form on $\so(2,1)$, up to a factor:
\begin{equation} \label{eq killing mink}
\kappa(\Lambda({x}),\Lambda({x}'))=2\langle {x},{x}'\rangle~,
\end{equation}
where the Killing form $\kappa$ for $\so(2,1)$ is:
$$\kappa(\mathfrak a,\mathfrak a')=\tr(\mathfrak a\mathfrak a')~.$$ 
\end{itemize}

\section{Balanced geodesic graphs on hyperbolic surfaces} \label{sec balanced graph}

In this section we introduce balanced geodesic graphs on a closed hyperbolic surface, we show that there is a vector space structure on the space of such objects, and we construct an element in $T_{[h]}\mathcal F(S)$ from any balanced geodesic graph on $(S,h)$. For convenience, we will use $\Hyp^2$ in the hyperboloid model defined in the previous section, and hence we will identify $\Isom(\Hyp^2)=\SO_0(2,1)$ and $\isom(\Hyp^2)=\so(2,1)$.

\subsection*{Balanced geodesic graphs}

Let us fix a hyperbolic metric $h$ on the closed oriented surface $S$.

\begin{defi} \label{defi geod graph}
A \emph{balanced geodesic graph} on $(S,h)$ is the datum of
\begin{itemize}
\item A finite embedded graph $\mathcal{G}=(\mathcal E,\mathcal V)$ in $S$, where $\mathcal E$ is the set of (unoriented) edges of $\mathcal{G}$ and $\mathcal V$ is the set of vertices;
\item The assignment $\mathbf w:\E\to\R$ of weights $\mathbf w(e)=w_e$ for every $e\in\mathcal E$;
\end{itemize}
satisfying the following conditions:
\begin{itemize}
\item Every edge $e$ is a geodesic interval between its endpoints;
\item For every $p\in\mathcal V$, the following \emph{balance condition} holds:
 \begin{equation}\label{eq:mink sum} {\sum_{p\in e} w_{e} v_{e}}=0~, \end{equation}
 where $p\in e$ denotes that $p$ is an endpoint of an edge $e\in \E$, and in this case $v_{e}$ is the unit tangent vector at $p$ to the geodesic edge $e$.
\end{itemize}
\end{defi}

We provide several classes of examples which should account for the abundance of such objects on any surface $(S,h)$. 

\begin{example} \label{example simple closed geo}
Given a simple closed geodesic $c$ on $(S,h)$ and a weight $w$, $c$ is the support of a balanced geodesic graph with weight $w$. In fact, it suffices to declare that the vertex set $\mathcal V$ consists of $n\geq 1$ points on $c$. Then $c$ is split into $n$ edges, and we declare that each edge has weight $w$. Clearly the balance condition  \eqref{eq:mink sum} is satisfied, since there are only two vectors to consider at every vertex, opposed to one another, with the same weight. See the curve in the left of Figure \ref{fig:closed geodesics}. Hence the class of balanced geodesic graphs include weighted simple closed geodesics.
\end{example}

\begin{example} \label{ex closed multigeodesics}
More generally, given any finite collection $c_1,\ldots,c_n$ of (not necessarily simple) closed geodesics $(S,h)$, and any choice of weights $w_1,\ldots,w_n$, the union $\cup_i c_i$ can be made into a balanced geodesic graph. In fact, it suffices to choose the vertices of $\V$ on the geodesics $c_1,\ldots,c_n$, so that every intersection point between $c_i$ and some other geodesic $c_j$ (including self-intersections of $c_i$) is in the vertex set $\mathcal V$. Moreover, it is necessary to add vertices to geodesics $c_i$ which are disjoint from all the geodesics $c_j$ (in particular, $c_i$ has no self-intersection) so as to make every edge of the graph an interval, as in Example \ref{example simple closed geo}. Then we declare that the weight of an edge contained in $c_i$ is $w_i$. In fact, in this case the balance condition will be automatically satisfied, since at every vertex $p\in\V$, tangent vectors come in opposite pairs with the same weight. (A pair is composed of the two opposite vectors tangent  to the same geodesic $c_i$.) See Figure \ref{fig:closed geodesics}. Therefore the balance condition \eqref{eq:mink sum} is satisfied regardless of the initial choice of $w_i$. In particuar, \emph{weighted multicurves} (i.e. collections of disjoint simple closed geodesics endowed with positive weights) are balanced geodesic graphs.


\begin{figure}
\begin{center}
\psfrag{=}{$=$}
\psfrag{+}{$+$}
\psfrag{b}{$w_e$}
\psfrag{b2}{$w'_e$}
\psfrag{b12}{$w_e+w'_e$}
\includegraphics[scale=0.8]{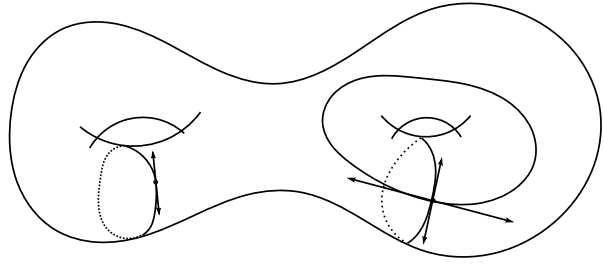}
\end{center}\caption{The balance condition is trivially satisfied in the cases of closed geodesics, as in Examples \ref{example simple closed geo} and \ref{ex closed multigeodesics}.\label{fig:closed geodesics}}
\end{figure}

\end{example}


\begin{example}
Colin de Verdi\`ere in \cite{cdv} proved that, given any topological triangulation of $(S,h)$, and any choice of positive weights $w_e$ assigned to each edge $e$, there exists a geodesic triangulation, with the same combinatorics of the original triangulation, which satisfies the balance condition \eqref{eq:mink sum} for the prescribed weights $w_e$ divided by the length of $e$. Therefore, this geodesic realization of a topological triangulation is a balanced geodesic graph in the sense of Definition \ref{defi geod graph}.

\end{example}

\subsection*{Vector space structure}
In this subsection we will introduce the space of balanced geodesic graphs on a hyperbolic surface, and show that this space has a vector space structure.

Let us consider the space
$$\BG_{(S,h)}:=\{\text{balanced geodesic graphs on }(S,h)\}/\sim~,$$
where the equivalence relation $\sim$ is defined as follows: two balanced geodesic graphs are equivalent if they can be obtained from one another by adding, or deleting:
\begin{itemize}
\item Points which are endpoints of only two edges (possibly coincident);
\item Edges of weight zero.
\end{itemize}
The space $\BG_{(S,h)}$ defined in this way is naturally endowed with a structure of vector space, defined in the following way. 
For 
$(\mathcal{G},\mathbf w)$ in $\BG_{(S,h)}$ and $\lambda\in \R$, we define
$\lambda (\mathcal{G},\mathbf w)=(\mathcal{G},\lambda \mathbf w)$. For the addition, we define  $(\mathcal{G},\mathbf w)\textbf{+}(\mathcal{G}',\mathbf w')= (\mathcal{G}+\mathcal{G}',\mathbf w+\mathbf w')$, where 
\red{\begin{itemize}
 \item $\mathcal{G}+\mathcal{G}'=(\mathcal E'',\mathcal V'')$ if  $\mathcal{G}=(\mathcal E,\mathcal V)$ and $\mathcal{G}'=(\mathcal E',\mathcal V')$, where  $\mathcal V''=\mathcal V \cup \mathcal V' \cup \bar{\mathcal V}$, and $\bar{\mathcal{V}}$ is the set of intersecting points of $\mathcal E$ and $\mathcal E'$ (may be empty), and $\mathcal E''$ is the set of edges between elements of $\mathcal V''$ which are subsets of elements of $\mathcal E\cup \mathcal E'$;
\item $\mathbf w+\mathbf w'$ are the weights on $\mathcal E''$ defined as follows. For $e\in \mathcal E$:
\begin{itemize}
\item if $e$ does not meet any element of $\mathcal E'$, then the same weight is kept;
\item if $e$ meets an element of $\mathcal E'$ creating a new vertex (an element of $\bar{\mathcal{V}}$),  
new weights are assigned as in the top of Figure~\ref{fig:sum w};
\item if $e$ meets an element of $\mathcal E'$ without creating a new vertex, new weights are assigned as in the bottom of Figure~\ref{fig:sum w}.
\end{itemize}
This procedure is done for every element of $\mathcal E$, and eventually the same weights are kept for elements of $\mathcal E'$ that don't meet any element of $\mathcal E$.
\end{itemize}}
It is easy to check that \eqref{eq:mink sum} is satisfied in every case. 


\begin{figure}
\begin{center}
\psfrag{=}{$=$}
\psfrag{+}{$+$}
\psfrag{b}{$w_e$}
\psfrag{b2}{$w'_e$}
\psfrag{b12}{$w_e+w'_e$}
\includegraphics[scale=1]{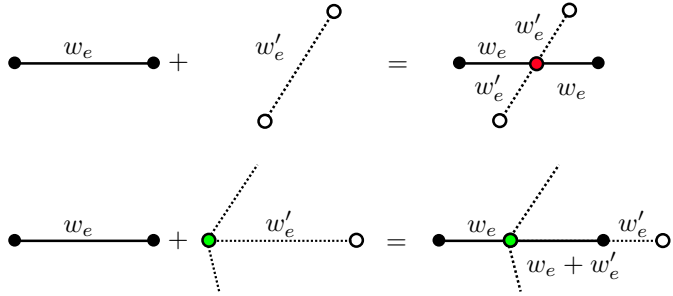}
\end{center}\caption{New assignements of weights for the \red{addition} of two balanced geodesic graphs.}\label{fig:sum w}
\end{figure}

Finally, the zero element of the vector space is the class of any geodesic graph whose weights are all zero. Let us observe that, for every (fixed) geodesic graph $\mathcal{G}_0$, the subset of $\BG_{(S,h)}$ composed of classes of balanced geodesic graphs having underlying geodesic graph $\mathcal{G}_0$ is a finite-dimensional vector subspace.

\subsection*{Construction of the map $\Phi$}

We are now ready to define the map from the space $\BG_{(S,h)}$ to the tangent space of Teichm\"uller space. We will actually define a deformation of the holonomy representation of $h$, thus providing a tangent vector to the character variety of $\pi_1(S)$. Namely, we will define a map:
$$\Phi:\BG_{(S,h)}\to H^1_{\Ad\rho}(\pi_1(S),\so(2,1))~,$$
where $[\rho]=\Hol([h])$. For this purpose, consider the universal cover  $\pi:\widetilde S\to S$ and fix a developing map 
$$\dev:(\widetilde S,\pi^*h)\to\Hyp^2$$
which is a global isometry, equivariant for the representation $\rho:\pi_1(S)\to \SO_0(2,1)$. 

Let us fix a balanced geodesic graph $(\mathcal{G},\mathbf w)$ and lift it to $(\widetilde{\mathcal{G}},\widetilde{\mathbf w})$ on  $(\widetilde S,\pi^*h)$. Let us also fix a basepoint $x_0\in\widetilde S$, which we assume does not lie in $\widetilde{\mathcal{G}}$. We define a cocycle $\tau:\pi_1(S)\to\so(2,1)$ in the following way. We say a path $\sigma:[0,1]\to\widetilde S$ is \emph{transverse} to $\widetilde{\mathcal{G}}$ if the intersection $\mathrm{Im}\sigma\cap\widetilde{\mathcal{G}}$ of the image of $\sigma$ and $\widetilde{\mathcal{G}}$ consists of a finite number of points, which are not vertices of $\widetilde{\mathcal{G}}$, and for every point $p=\sigma(t_0)\in\mathrm{Im}\sigma\cap\widetilde{\mathcal{G}}$ there exists $\epsilon>0$ such that $\sigma|_{(t_0-\epsilon)}$
and $\sigma|_{(t_0+\epsilon)}$ are contained in two different connected component of $U\setminus \widetilde{\mathcal{G}}$, where $U$ is a small neighborhood of $p$ in $\widetilde S$.

Now, pick $\gamma\in\pi_1(S)$ and let $\sigma:[0,1]\to\widetilde S$ be a path such that $\sigma(0)=x_0$ and $\sigma(1)=\gamma\cdot x_0$, transverse to $\widetilde{\mathcal{G}}$. Then we define the following element of $\so(2,1)$:
\begin{equation} \label{defi cocycle}
\tau(\gamma)=\sum_{p\in \mathrm{Im}\sigma\cap \widetilde e} \widetilde w_{\widetilde e} \mathfrak t(\dev(\widetilde e))~,
\end{equation}
where:
\begin{itemize}
\item The sum is over all points $p$ of intersection of the image of the path $\sigma$ with the lift $\widetilde{\mathcal{G}}$ of the balanced geodesic graph $\mathcal{G}$, see Figure \ref{fig:chemin};
\item \red{We orient $\tilde{e}$ on the left when it is crossed by $\sigma$. Hence $\dev(\widetilde e)$ is an oriented (subinterval of) a geodesic $\ell$ of $\Hyp^2$, and } recall from Section \ref{sec hyperboloid} that we denote $\mathfrak t(\ell)\in\so(2,1)$ the infinitesimal hyperbolic translation along $\ell$, using the orientation of $\ell$. Namely,
$$\mathfrak t(\ell)=\left.\frac{d}{dt}\right|_{t=0}T_t(\ell)$$
where $T_t(\ell)\in\SO_0(2,1)$ is the isometry of $\Hyp^2$ which preserves $\ell$ and translates every point of $\ell$ by a length
 $t$ according to the orientation of $\ell$. This is thus applied in Equation \eqref{defi cocycle} \red{with $\mathfrak t(\dev(\widetilde e))=\mathfrak t(\ell)\in \so(2,1)$}. 

\red{ Equivalently, $\mathfrak t(
\dev(\widetilde e))=\Lambda(x)$ where $\dev(\widetilde e)\subset x^\perp$, and $x$ is pointing in the same halfspace than the tangent vector of $\dev(\sigma)$ (seen as a space-like vector in $\R^{2,1}$) when it crosses  $\dev(\widetilde e)$.  } See Figure \ref{fig:chemin2}.
\item The coefficient $\widetilde w_{\widetilde e}$ in the sum equals the weight $w_e$ of the edge
$e=\pi(\widetilde e)$. 
\end{itemize}
We then define 
$$\Phi(\mathcal{G},\mathbf w)=[\tau]\in H^1_{\Ad\rho}(\pi_1(S),\so(2,1))~.$$
Under the identification between $H^1_{\Ad\rho}(\pi_1(S),\so(2,1))$ and $T_{[X]}\Teich(S)$, we thus define 
\begin{equation} \label{eq defi variation teich}
t_{(\mathcal{G},\mathbf w)}=d(\Hol\circ U)^{-1}\circ \Phi(\mathcal{G},\mathbf w)~.
\end{equation}

\begin{figure}
\begin{center}
\psfrag{x}{$x_0$}
\psfrag{gx}{$\gamma\cdot x_0$}
\psfrag{s}{$\sigma$}
\psfrag{p}{$p$}
\psfrag{el}{$\tilde{e}$}
\includegraphics[scale=1]{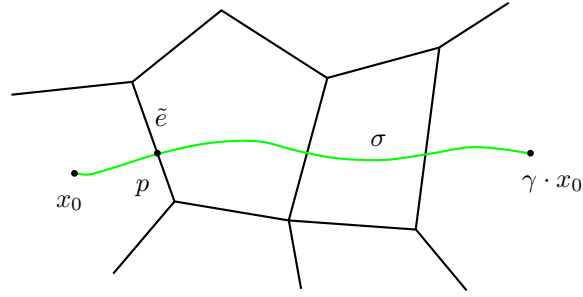}
\end{center}\caption{To define $\Phi$, we lift a generic closed loop representing $\gamma\in\pi_1(S)$ to a path $\sigma$ in $\tilde{S}$, and consider a sum over the points of intersection with lifted edges $\widetilde e$.}\label{fig:chemin}
\end{figure}

\begin{figure}
\begin{center}
\psfrag{si}{$\operatorname{dev}(\sigma)$}
\psfrag{el}{$\ell$}
\psfrag{dev}{$\operatorname{dev}(\tilde{e})$}
\includegraphics[scale=0.8]{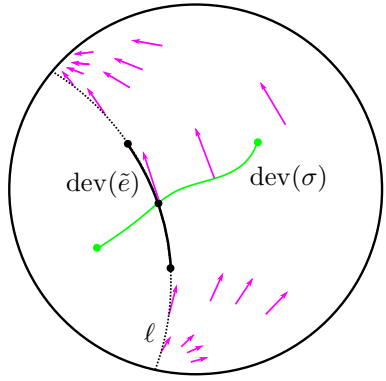}
\end{center}\caption{The image $\mathbb{H}^2$ (in the Poincar\'e disc model) of the path $\sigma$. The geodesic $\ell$ contains the development of the edge $\widetilde e$. The arrows represent some vectors of the Killing field $\mathfrak{t}(\operatorname{dev}(\tilde{e}))$.}\label{fig:chemin2}
\end{figure}

\subsection*{Well-definiteness of $\Phi$}

There is a number of points to be verified in order to check that the map $\Phi$ is well-defined. First, we need to show that the value $\tau(\gamma)$ does not depend on the chosen path $\sigma:[0,1]\to\widetilde S$, as long as $\sigma$ is transverse to $\widetilde{\mathcal{G}}$. (This is the same as choosing a representative of the closed loop in $S$, based at $\pi(x_0)$, which represents $\gamma$ and is transverse to $\mathcal{G}$.) 

In fact, there are three cases to consider.  See also Figure \ref{fig:3cases}.
\begin{enumerate}
\item If two representatives $\sigma$ and $\sigma'$ can 
be isotoped to one another by a family of paths $\sigma_t:[0,1]\to\widetilde S$ which is transverse for all $t$, then the value of $\tau$ is the same when computed with respect to $\sigma$ or $\sigma'$, since the quantities $\widetilde w_{\widetilde e}$ and  $\mathfrak t(\dev(\widetilde e))$ only depend on the edge $\widetilde e$, and not on the intersection point of $\sigma$ with $\widetilde e$. 
\item Suppose $\sigma$ and $\sigma'$ agree on the complement of a small neighborhood $U_p$ of a vertex $p\in\mathcal V$, and consider an isotopy for $\sigma$ and $\sigma'$ which crosses $p$ at some time $t_0\in(0,1)$ and is constant in $\widetilde S\setminus U_p$. Observe that the balance condition \eqref{eq:mink sum} is equivalent to the following condition:
\begin{equation} \label{eq balance condition lie algebra}
\sum_{p\in \widetilde e} \widetilde w_{\widetilde e} \mathfrak t(\dev(\widetilde e))=0~.
\end{equation}
Indeed, from Equation \eqref{eq:mink sum}, by lifting to the universal cover and rotating all vectors by $\pi/2$, one obtains 
$${\sum_{p\in \widetilde e} \widetilde w_{\widetilde e} y_{\widetilde e}}=0~,$$
where $y_{\widetilde e}$ is the unit vector orthogonal to the geodesic containing the edge $\widetilde e$, and can thus be interpreted as a unit spacelike vector in $\R^{2,1}$.
This shows that the result for $\tau$, defined in Equation \eqref{defi cocycle}, is the same if computed using $\sigma$ or $\sigma'$.
\item Finally, suppose $\sigma$ and $\sigma'$ only differ in such a way that one of the two paths intersects the same edge $\widetilde e$ at two consecutive points while the other does not. Then the contributions given by such consecutive intersections cancel out, hence the result is again the same. 
\end{enumerate}


\begin{figure}
\begin{center}
\psfrag{s}{$\sigma$}
\psfrag{s2}{$\sigma'$}
\includegraphics[scale=1]{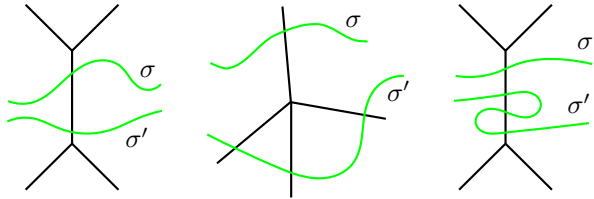}
\end{center}\caption{The three cases in the proof of well-definiteness of $\Phi$.}\label{fig:3cases}
\end{figure}

Since every two transverse paths connecting $x_0$ and $\gamma\cdot x_0$ can be deformed to one another by a sequence of moves of the three above types, we have shown that the definition of $\tau(\gamma)$ does not depend on the choice of the path $\sigma$ representing $\gamma$.

Second, we need to check $\tau\in Z^1_\rho(\pi_1(S),\so(2,1))$, that is, $\tau$ satisfies the cocycle condition of Equation \eqref{eq cocycle condition}. In fact, let $\sigma_1$ be a transverse path connecting $x_0$ and $\gamma\cdot x_0$ as in the definition, and similarly let $\sigma_2$ be a transverse path connecting $x_0$ and $\eta\cdot x_0$. To represent $\eta\gamma$, we can use the concatenation of $\sigma_1$ and $\eta\cdot \sigma_2$. Let $\sigma$ be such a concatenation of paths. Then we have
\begin{equation} \label{eq computation cocycle}
\tau(\gamma\eta)=\sum_{p\in \mathrm{Im}\sigma_1\cap \widetilde e} \widetilde w_{\widetilde e} \mathfrak t(\dev(\widetilde e))+\sum_{p\in \mathrm{Im}(\eta\cdot\sigma_2)\cap \widetilde e} \widetilde w_{\widetilde e} \mathfrak t(\dev(\widetilde e))~.
\end{equation}
Now, the first term in the summation is $\tau(\gamma)$, while the second term equals
$$\sum_{p\in \mathrm{Im}\sigma_2\cap \widetilde e} \widetilde w_{\widetilde e} \mathfrak t(\dev(\eta\cdot\widetilde e))=\sum_{p\in \mathrm{Im}\sigma_2\cap \widetilde e} \widetilde w_{\widetilde e} \mathfrak t(\rho(\eta)\dev(\widetilde e))=\sum_{p\in \mathrm{Im}\sigma_2\cap \widetilde e} \widetilde w_{\widetilde e} \Ad\rho(\eta)\cdot\mathfrak t(\dev(\widetilde e))~,$$
where we have applied the equivariance of $(\widetilde{\mathcal{G}},\widetilde{\mathbf w})$ and of $\dev$ for the holonomy representation $\rho$, and the property that the infinitesimal hyperbolic translation along the geodesic $\rho(\eta)\cdot \ell$ coincides with the infinitesimal translation along $\ell$ composed with $\Ad\rho(\eta)$. This shows that the second term in \eqref{eq computation cocycle} coincides with $\Ad\rho(\eta)\cdot \tau(\gamma)$ and thus concludes the claim.

It now only remains to show that the definition of $\Phi$ does not depend on the choice of the basepoint $x_0$. In fact, given another basepoint $x_0'$, let $\sigma'$ be a path connecting $x_0'$ and $\gamma\cdot x_0'$ and let $\tau'(\gamma)$ be obtained by applying the definition \eqref{defi cocycle} to $\sigma'$. Then one has
$$\tau'(\gamma)-\tau(\gamma)=\Ad\rho(\gamma)\cdot \tau_0-\tau_0~,$$
where $\tau_0$ is the quantity obtained by a summation, exactly as in \eqref{defi cocycle}, along a transverse path which connects $x_0$ and $x_1$. This shows that $[\tau]$ is well-defined in $H^1_{\Ad\rho}(\pi_1(S),\so(2,1))$.

\red{\begin{remark}
We remark that, in the case the balanced geodesic graph is a weighted multicurve, our construction recovers a tangent vector on Teichm\"uller space which is an infinitesimal left earthquake. In fact, the piecewise Killing vector field we construct is, on each stratum, nothing but the (infinitesimal) displacement with respect to the stratum containing $x_0$, which is fixed by construction. The fact that the definition of $\Phi$ does not depend on the choice of the basepoint nor on the path $\sigma$ essentially reflect the fact that for earthquakes on surfaces, relative displacements are well-defined, once one stipulates that all earthquakes are left.
\end{remark}}

\section{Geometric description}

In this section we will give two types of interpretations of the infinitesimal deformation $\Phi(\mathcal{G},\mathbf w)$ we have produced out of a balanced geodesic graph. The first interpretation should be interpreted as a motivation, since it generalizes infinitesimal twist along simple closed geodesics. The second concerns polyhedral surfaces in Minkowski space, and is then applied to show that $\Phi$ is surjective. Some details are omitted, since investigation of these viewpoints is beyond the scope of this paper and is thus left for future work.

\subsection*{Infinitesimal twist along simple closed geodesics} 
Let us consider a simple closed geodesic $c$ on $(S,h)$. As in Example \ref{example simple closed geo}, one can turn $c$ into a balanced geodesic graph $\mathcal{G}_c$, with constant weight $w$. If we put $w=1$, then $\Phi(\mathcal{G}_c,1)$ is the \emph{infinitesimal left twisting}, or \emph{infinitesimal left earthquake}, along the simple closed geodesic $c$, see Proposition~B.3  in \cite{bonschlfixed}. This is exactly the object which appears in Wolpert's formula in the articles \cite{wolpertformula} and \cite{wolpertelementary}.

\subsection*{Flippable tilings on hyperbolic surfaces}
More generally, let us assume $\mathcal{G}$ is a geodesic graph, which disconnects $S$ in convex geodesic faces. Then there are (differentiable) deformations $h_t$ of the hyperbolic metric $h$ so that $h_t$ contains a geodesic graph $\mathcal{G}_t$ which is a \emph{left flippable tiling} in the sense of \cite{flippable} (which is the reference to be consulted for more details). Roughly speaking, this means that the faces of $\mathcal{G}_t$ can be divided into \emph{black faces} and \emph{white faces}, and the black faces can be \emph{flipped} to obtain a new hyperbolic metric $h_t'$. The metric $h_t'$ is also endowed with a geodesic graph, which is a \emph{right flippable tiling}. As $t\to 0$, the metrics $h_t$ and $h_t'$ converge to the original metric $h$. The black faces of $\mathcal{G}_t$ and $\mathcal{G}_t'$ collapse continuously to the vertices of the original graph $\mathcal{G}$, while the white faces converge to the connected components of $S\setminus \mathcal{G}$. Moreover, the derivatives of the lengths of the edges of the black faces satisfy the balance condition, and thus define (positive) weights $\mathbf w$ such that 
 $(\mathcal G,\mathbf w)$ is a balanced geodesic graph. See Figure~\ref{fig:flip}.

A deformation of $(\mathcal G,\mathbf w)$ by left and right flippable tilings is not canonical, but by a direct computation one can show that their difference at first order is uniquely determined by $(\mathcal G,\mathbf w)$, and coincides with the quantity $\Phi(\mathcal{G},\mathbf w)$ we defined. Namely,
$$\Phi(\mathcal{G},\mathbf w)=\frac{1}{2}\left(\left.\frac{d}{dt}\right|_{t=0}[\rho_t]-\left.\frac{d}{dt}\right|_{t=0}[\rho'_t]\right)~,$$
where $\rho_t$ is the holonomy of  $h_t$ and $\rho'_t$ is the holonomy of $h'_t$.

\begin{figure}
\begin{center}
\psfrag{1b}{$1'$}\psfrag{1}{$1$}
\psfrag{2b}{$2'$}\psfrag{2}{$2$}
\psfrag{3b}{$3'$}\psfrag{3}{$3$}
\psfrag{4b}{$4'$}\psfrag{4}{$4$}
\psfrag{5b}{$5'$}\psfrag{5}{$5$}
\psfrag{6b}{$6'$}\psfrag{6}{$6$}
\psfrag{p}{$p$}
 \psfrag{a}{\textcolor{white}{$a$}}
 \psfrag{b}{\textcolor{white}{$b$}}
  \psfrag{c}{\textcolor{white}{$c$}}
   \psfrag{d}{\textcolor{white}{$d$}}
    \psfrag{e}{\textcolor{white}{$e$}}
   \psfrag{f}{\textcolor{white}{$f$}}
\includegraphics[scale=1]{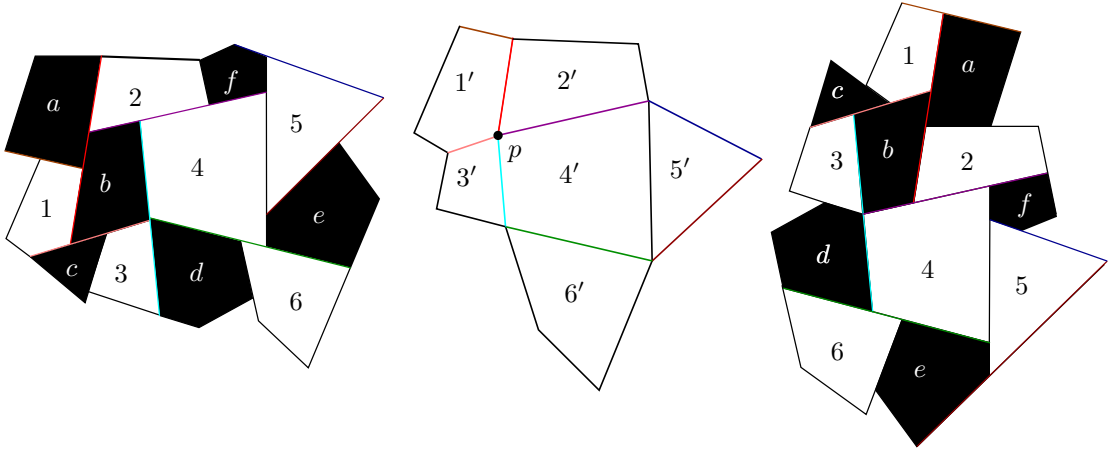}
\end{center}\caption{The left hand side of the picture is a piece of a left flippable tiling on a hyperbolic surface. The picture on the right hand side is a right flippable tiling instead, obtained from the left one by a flip. In particular, the corresponding polygons are isometric. The tilings are deformed such that the length of the edges of the black faces goes to zero and the resulting 
tiling of white faces gives the hyperbolic metric $(S,h)$. If the edges around $p$ are weighted by the derivatives of the lengths of the black faces, then the balance condition \eqref{eq balance intro}  is satisfied.}\label{fig:flip}
\end{figure}

\subsection*{Polyhedral surfaces in Minkowski space}
Let us now move to the second interpretation of the map $\Phi$. Let $\rho:\pi_1(S)\to\SO_0(2,1)$ be the holonomy representation of a hyperbolic surface $(S,h)$ and let $[\tau]\in H^1_{\Ad\rho}(\pi_1(S),\so(2,1))$. Using the isomorphism $\Lambda:\R^{2,1}\to\so(2,1)$ introduced at the end of Section \ref{sec hyperboloid}, and the equivariance of $\Lambda$ for the natural $\SO_0(2,1)$-actions, we have an isomorphism, which we still denote by $\Lambda$, between the spaces of cocycles:
$$\Lambda:{Z}^1_{\rho}(\pi_1(S),\R^{2,1})\to {Z}^1_{\Ad\rho}(\pi_1(S),\so(2,1))~.$$
Analogously, it induces an isomorphism between the spaces of coboundaries, and thus we have an isomorphism
$$[\Lambda]:H^1_{\rho}(\pi_1(S),\R^{2,1})\to H^1_{\Ad\rho}(\pi_1(S),\so(2,1))~.$$

Recall that the identity component of the isometry group of $\R^{2,1}$ is isomorphic to
$$\SO_0(2,1)\rtimes\R^{2,1}~,$$
and that, given a representation $\widehat\rho:\pi_1(S)\to \SO_0(2,1)\rtimes\R^{2,1}$ with linear part $\rho$, the translation part of such representation is a cocycle in $Z^1_{\rho}(\pi_1(S),\R^{2,1})$.
Now, let $[\tau]\in H^1_{\Ad\rho}(\pi_1(S),\so(2,1))$ and suppose $\Sigma$ is a
convex polyhedral surface in $\R^{2,1}$ with spacelike faces, invariant by the action of $\widehat\rho(\pi_1(S))$, where the linear part  is $\rho(\pi_1(S))$ and the translation part is $\Lambda(\tau)$. (This corresponds to the lift to $\R^{2,1}$ of a convex polyhedral surface in a maximal globally hyperbolic flat three-manifold, studied from this point of view in \cite{mess}.) It then turns out that the Gauss map $G:\Sigma\to\Hyp^2$ is a set-valued map equivariant with respect to the action of $\widehat\rho$ on $\Sigma$, and of $\rho$ on $\Hyp^2$. The image of the faces of $\Sigma$ are points in $\Hyp^2$, the image of edges of $\Sigma$ are geodesic edges. Therefore $G$ defines a geodesic graph $\mathcal{G}$ on $(S,h)$.

Moreover, we can define a set of weights $\mathbf w$ on $\mathcal{G}$. Given an edge $e$ of $\mathcal{G}$, let $\widetilde e$ be any lift of $e$ to $\Hyp^2$, and let $\widetilde e'=G^{-1}(\widetilde e)$ be the corresponding edge of the polyhedral surface $\Sigma$. Then we define $w_e$ as the length (for the Minkowski metric) of $\widetilde e'$, which clearly does not depend on the chosen lift. It can be easily seen that $(\mathcal{G},\mathbf w)$ is balanced, since the balance condition \eqref{eq:mink sum} is exactly equivalent to the condition that the faces of $\Sigma$ are bounded by a set of edges which ``close up'' in $\R^{2,1}$. 

Finally, let us remark that the balanced geodesic graph constructed in this way has positive weights on all edges. Moreover, the connected components of $S\setminus \mathcal G$ are convex geodesic polygons.

\begin{remark} \label{remark phi holonomy}
From the above construction, if $(\mathcal G,\mathbf w)$ is the balanced geodesic graph on $(S,h)$ associated to an invariant convex polyhedral surface $\Sigma$, then the translation part of the representation $\widehat \rho:\pi_1(S)\to \SO_0(2,1)\rtimes\R^{2,1}$ is precisely $[\Lambda]^{-1}\circ\Phi(\mathcal G,\mathbf w)$. {See also Section~4.4 in \cite{Fillastre} and the references therein}.
\end{remark}

This construction shows the abundance of examples of balanced geodesic graphs on a closed hyperbolic surface. In fact, 
it enables to show the following.

\begin{prop}
The map $\Phi:\BG_{(S,h)}\to H^1_{\Ad\rho}(\pi_1(S),\so(2,1))$ is surjective.
\end{prop}
\begin{proof}
Given any $[\tau]\in H^1_{\Ad\rho}(\pi_1(S),\so(2,1))$, by the work of Mess (\cite{mess}), the image of the representation $\widehat \rho:\pi_1(S)\to \SO_0(2,1)\rtimes\R^{2,1}$ with linear part $\rho$ and translation part $\Lambda^{-1}(\tau)$ acts freely and properly discontinously on a future-convex domain $\mathcal D$ in $\R^{2,1}$. Then take a finite number of points $p_i$ in a fundamental domain for the action of $\widehat \rho(\pi_1(S))$ on $\mathcal D$. The convex hull of the orbit $\cup_i \widehat \rho(\pi_1(S))\cdot p_i$ is a future-convex domain in $\R^{2,1}$ (actually, contained in $\mathcal D$), whose boundary is a $\widehat \rho(\pi_1(S))$-invariant spacelike convex polyhedral surface. Let $(\mathcal G,\mathbf w)$ be the balanced geodesic graph on $(S,h)$ associated to $\Sigma$, as constructed above. By Remark \ref{remark phi holonomy}, $\Phi(\mathcal G,\mathbf w)=[\tau]$. This concludes the proof.
\end{proof}

\section{De Rham cohomology} \label{sec de rham}

In this section we introduce the final tool needed to prove our main theorem, namely de Rham cohomology with values in a certain vector bundle over $(S,h)$. We then conclude the proof of the generalization of Wolpert's formula. 

\subsection*{Flat vector bundles}
Given a hyperbolic surface $(S,h)$, let $\rho:\pi_1(S)\to\SO_0(2,1)$ be its holonomy representation. Let us consider the flat vector bundle $F_\rho$ defined in the following way. Let $\widetilde S$ be the universal cover of $S$, and let 
$$F_\rho=(\widetilde S\times \so(2,1))/\pi_1(S)~,$$
where $\pi_1(S)$ is acting on $\widetilde S\times \so(2,1)$ by 
$$\gamma\cdot (p,x)=(\gamma\cdot p,\Ad\rho(\gamma)\cdot x)~.$$
Then $F_\rho$ is endowed with a flat vector bundle structure with base the original surface $(S,h)$. Hence we can consider de Rham cohomology with values in $F_\rho$, namely
$$H^1_{\rham}(S,F_\rho)=\frac{Z^1_{\rham}(S,F_\rho)}{B^1_{\rham}(S,F_\rho)}~,$$
where $Z^1_{\rham}(S,F_\rho)$ is the vector space of closed $F_\rho$-valued 1-forms and $B^1_{\rham}(S,F_\rho)$ the subspace of exact 1-forms.

It is well-known that there exists a natural vector space isomorphism
$$\Psi:H^1_{\rham}(S,F_\rho)\to H^1_{\Ad\rho}(\pi_1(S),\so(2,1))~.$$
Nevertheless, we review for convenience of the reader the construction of such vector space isomorphism, as it will be useful to prove our main theorem. First, observe that $F_\rho$-valued closed 1-forms $\alpha$ are in 1-1 correspondence with closed $\so(2,1)$-valued 1-forms $\widetilde \alpha$ on $\widetilde S$ satisfying the equivariance 
\begin{equation} \label{eq equivariance 1-form}
\widetilde \alpha\circ \gamma_*=\Ad\rho(\gamma)\circ \widetilde\alpha~.
\end{equation}
Hence, given a $F_\rho$-valued closed 1-form $\alpha$, we define a cocycle $\tau_\alpha$ by
$$\tau_\alpha(\gamma)=\int_{x_0}^{\gamma\cdot x_0}\widetilde\alpha:=\int_\sigma \widetilde\alpha~,$$
where $\sigma:[0,1]\to\widetilde S$ is a smooth path such that $\sigma(0)=x_0$ and $\sigma(1)=\gamma\cdot x_0$. Then we define
$$\Psi([\alpha]_\rham)=[\tau_\alpha]\in H^1_{\Ad\rho}(\pi_1(S),\so(2,1))~.$$
Since $\alpha$ is closed (hence also $\widetilde \alpha$) and $\widetilde S$ is simply connected, $\tau_\alpha(\gamma)$ does not depend on the choice of the path $\sigma$ with fixed endpoints. This authorizes us to use the notation $\int_{x_0}^{\gamma\cdot x_0}\widetilde\alpha$. By an argument analogous to that of Section \ref{sec balanced graph}, changing the basepoint $x_0$ results in changing $\tau_\alpha$ by a coboundary. 

Moreover, if $\alpha=df$ is exact, where $f$ is a global section of $F_\rho$, then $\widetilde\alpha=d\widetilde f$ and $\widetilde f$ satisfies 
$$\widetilde f\circ \gamma=\Ad\rho(\gamma)\cdot \widetilde f~.$$
 Hence in this case, by the fundamental theorem of calculus, 
 $$\tau_\alpha=\int_{x_0}^{\gamma\cdot x_0}d\widetilde f=\widetilde f(\gamma\cdot x_0)-\widetilde  f(x_0)=\Ad\rho(\gamma)\cdot \widetilde f(x_0)-\widetilde f(x_0)$$ is a coboundary as in \eqref{eq defi coboundary}. This shows that the map $\alpha\mapsto\tau_\alpha$ induces a well-defined map from $H^1_{\rham}(S,F_\rho)$ to $H^1_{\Ad\rho}(\pi_1(S),\so(2,1))$. 
 
As a warm-up, and for convenience of the reader, we also show that $\Psi$ is injective.
\begin{lemma}
The linear map $\Psi:H^1_{\rham}(S,F_\rho)\to H^1_{\Ad\rho}(\pi_1(S),\so(2,1))$ is injective.
\end{lemma}
\begin{proof}
Suppose the image of $\alpha$ is trivial in the cohomology $H^1_{\Ad\rho}(\pi_1(S),\so(2,1))$. Since $\widetilde S$ is simply connected, there exists a function $g:\widetilde S\to\so(2,1)$ such that $\widetilde \alpha=dg$. Since $[\tau_\alpha]=0$, there exists $\xi\in\so(2,1)$ such that
 $$\tau_\alpha(\gamma)=g(\gamma\cdot x_0)-g(x_0)=\Ad\rho(\gamma)\cdot \xi-\xi~.$$
 
 Now choose a section $f_0\in H^0(S,F_\rho)$ such that its lift $\widetilde f_0:\widetilde S\to\so(2,1)$ at the point $x_0$ takes the value $\widetilde f_0(x_0)=\xi$. By this choice, $\tau_{df}(\gamma)=\Ad\rho(\gamma)\cdot \xi-\xi$, and thus 
 $\tau_{\alpha-df_0}=0$. This implies that
 \begin{equation} \label{eq vanishing integration}
 \int_{x_0}^{\gamma\cdot x_0}\widetilde\alpha-d\widetilde f_0:=\int_\sigma \widetilde\alpha-d\widetilde f_0=0~,
 \end{equation}
 where $\gamma\in\pi_1(S)$ and $\sigma$ is is any path connecting $x_0$ and $\gamma\cdot x_0$. 
 
 The proof will be concluded if we show that $\alpha-df_0$ is an exact 1-form. In fact, define $f_1:\widetilde S\to\so(2,1)$ by
 $$f_1(x)=\int_{x_0}^{x} \widetilde\alpha-d\widetilde f_0~.$$
 Clearly 
 $$\widetilde \alpha=d\widetilde f_0+d f_1~.$$ 
 If we show that $f_1$ induces a section of $F_\rho$, this will imply that $df_1$ induces an exact $F_\rho$-valued 1-form on $S$ and $\alpha=df_0+df_1$, and we will be done. In fact, we have
$$f_1(\gamma\cdot x)=\int^{\gamma\cdot x_0}_{x_0}\widetilde\alpha-d\widetilde f_0+\int^{\gamma\cdot x}_{\gamma\cdot x_0}\widetilde\alpha-d\widetilde f_0=\rho(\gamma) \int^{x}_{x_0}\widetilde\alpha-d\widetilde f_0=\rho(\gamma)\cdot f_1(x)~,$$
where we have used Equation \eqref{eq vanishing integration} and the equivariance of $\widetilde \alpha-d\widetilde f_0$. Hence $f_1$ comes from a section in $H^0(S,F_\rho)$ and this shows the claim.
\end{proof}

\subsection*{Reconstructing 1-forms from cocycles}
Although surjectivity of the map $H^1_{\rham}(S,F_\rho)\to H^1_{\Ad\rho}(\pi_1(S),\so(2,1))$
follows from injectivity and the fact that both vector spaces 
have dimension $3|\chi(S)|$, we will need an explicit construction for the inverse of this map.

Given $[\tau]\in H^1_{\Ad\rho}(\pi_1(S),\so(2,1))$, let $(\mathcal G,\mathbf w)$ be a balanced geodesic graph such that $\Phi(\mathcal G,\mathbf w)=[\tau]$. Let us denote $N_\epsilon(\mathcal G)$ the $\epsilon$-neighborhood of $\mathcal G$ for the hyperbolic metric $(S,h)$, and $B(p,\epsilon)$ the $h$-geodesic ball of radius $\epsilon$ centered at $p\in S$.  
Let us now pick $\epsilon>0$ such that $N_{3\epsilon}(\mathcal G)$ is embedded and the balls $B(p,3\epsilon)$ centered at the vertices $p\in\mathcal V$ are pairwise disjoint.

Now, define a $F_\rho$-valued closed 1-form $\alpha\in B^1_{\rham}(S,F_\rho)$ as follows. First, define 
$$\alpha=0\textrm{ on }S\setminus (\bigcup_{p\in\mathcal V}B(p,2\epsilon)\cup N_\epsilon(\mathcal G))~.$$
Second, for every edge $e$ of $\mathcal G$, let $\widetilde e$ be a lift of $e$ to $\widetilde S$, and put:
\begin{equation} \label{eq defi derham form}
\alpha=\left(w_e {\mathfrak t}(\dev(\widetilde e))\right) d(f\circ \delta)\textrm{ on }N_\epsilon(e)\setminus \bigcup_{p\in\mathcal V}B(p,2\epsilon)~,
\end{equation}
where:
\begin{itemize}
\item $\delta:N_\epsilon(\mathcal G)\to\R$ is the signed distance function $\delta(q)=d_{(S,h)}(q,\mathcal G)$;
\item $f:[-\epsilon,\epsilon]\to \R$ is a smooth function such that $f(\epsilon)=1/2$, $f(-\epsilon)=-1/2$, $f(-x)=-f(x)$, and the derivative $f'$ is compactly supported in $(-\epsilon,\epsilon)$; 
\item $\mathfrak t(\dev(\widetilde e))$ is the infinitesimal hyperbolic isometry along the geodesic containing $\dev(\widetilde e)$, where $\dev$ is the developing map of $(S,h)$ and $\widetilde e$ is a lift of the edge $e$. We choose $\mathfrak t(\dev(\widetilde e))$ to be the infinitesimal translation on the left, as seen from $\{\delta<0\}$ to $\{\delta>0\}$.
\end{itemize}
Observe that the definition of $\mathfrak t(\dev(\widetilde e))$ depends on the chosen orientation of 
$\dev(\widetilde e)$, which is implicitely induced by the choice of the sign of the distance function $\delta$. However, the quantity $\left(w_e {\mathfrak t}(\dev(\widetilde e))\right) d(f\circ \delta)$ is independent of such choice, since if one replaces $\delta$ by $-\delta$, ${\mathfrak t}(\dev(\widetilde e))$ takes the opposite sign, and also $d(f\circ \delta)$ changes sign since $f(-x)=-f(x)$.

Moreover, since $\dev(\gamma\cdot \widetilde e)=\rho(\gamma)\cdot \dev(\widetilde e)$, we have $\mathfrak t(\dev(\gamma\cdot \widetilde e))=\Ad\rho(\gamma)\cdot {\mathfrak{t}}(\dev(\widetilde e))$, and therefore $\alpha$ is well-defined as a 1-form with values in $F_\rho$ (which is only defined on the complement of the balls $B(p,2\epsilon)$ centered around vertices, for the moment). 

To conclude the definition, it remains to define $\alpha$ on the balls $B(p,2\epsilon)$ centered at vertices of $\mathcal V$. Before doing this, we give a remark.

\begin{remark} \label{remark integral derham}
For every path $\sigma:[0,1]\to \widetilde S$ transverse to $\widetilde{\mathcal G}$ (which we suppose has image in the complement of the $2\epsilon$-balls centered at vertices, for the moment), if $\sigma(0)$ and $\sigma(1)$ lie outside the $\epsilon$-neighborhood $N_\epsilon(\widetilde{\mathcal G})$, then the integral of $\widetilde\alpha$ over $\sigma$ is given by the sum of $w_e {\mathfrak{t}}(\dev(\widetilde e))$ over all edges $\widetilde e$ met by $\sigma$, and $e=\pi(\widetilde e)$. 
\end{remark}

Hence for every vertex $p$ of $\mathcal G,$
if $\gamma$ is any loop which generates the fundamental group of the annulus $B(p,3\epsilon)\setminus B(p,2\epsilon)$, $\int_\gamma \alpha=0$ as a consequence of the balance condition \eqref{eq balance condition lie algebra}.  
{It follows that $\alpha$ is exact on the annulus, hence it  can be smoothly extended  to a closed 1-form on the disc, and in turn on $S$.}
See Figure \ref{fig:extending inside balls}.


\begin{figure}
\begin{center}
\psfrag{p}{$p$}
\psfrag{B}{$B(p,3\epsilon)$}
\psfrag{N}{$N_\epsilon(\mathcal G)$}
\includegraphics[scale=1]{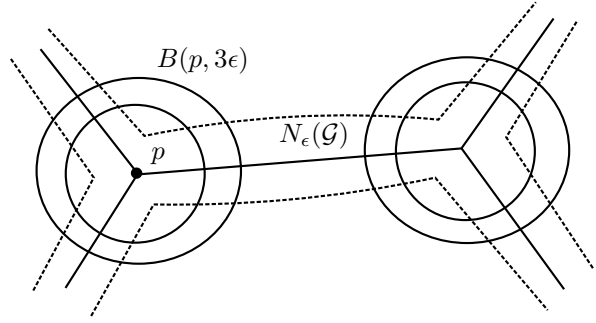}
\end{center}\caption{The 1-form $\alpha$, first defined in the complement of the balls of radius $2\epsilon$ around the vertices, can be extended smoothly inside the balls.
\label{fig:extending inside balls}}
\end{figure}

Using again Remark \ref{remark integral derham},
and comparing the definitions of the linear maps $\Psi:H^1_{\rham}(S,F_\rho)\to H^1_{\Ad\rho}(\pi_1(S),\so(2,1))$ and $\Phi:\BG_{(S,h)}\to H^1_{\Ad\rho}(\pi_1(S),\so(2,1))$, it follows that the 1-form $\alpha$ we have constructed satisfies $\Psi([\alpha]_\rham)=\Phi(\mathcal G,\mathbf w)=[\tau]$, namely, $\alpha$ realizes the cohomology class $[\tau]$ we started from.
Hence we provided a direct proof of the following:

\begin{lemma}
The linear map $\Psi:H^1_{\rham}(S,F_\rho)\to H^1_{\Ad\rho}(\pi_1(S),\so(2,1))$ is surjective.
\end{lemma}

\subsection*{The Goldman form in de Rham cohomology}

Following \cite{Goldman}, we can now express the Goldman form in terms of de Rham cohomology. In fact, using the isomorphism $\Psi:H^1_{\rham}(S,F_\rho)\to H^1_{\Ad\rho}(\pi_1(S),\so(2,1))$ we have just introduced, Goldman showed that 
the cup product in the group cohomology $H^1_{\Ad\rho}(\pi_1(S),\so(2,1))$ corresponds to the wedge operation between 1-forms  --- where in both cases we use the Killing form as a pairing. That is,
\begin{equation} \label{eq goldman form in de rham}
\omega_{\mathrm G}(\Psi[\alpha],\Psi[\alpha'])=\int_S \kappa(\alpha\wedge\alpha')~,
\end{equation}
where $\kappa(\alpha\wedge\alpha')$ denotes the wedge product paired by using the Killing form of $\so(2,1)$: if $X,Y$ are smooth vector fields on $S$, then 
$\kappa(\alpha \wedge \alpha')(X,Y)=\kappa(\alpha(X),\alpha'(Y))-\kappa(\alpha(Y),\alpha'(X))$.

\section{Proof of the generalized Wolpert formula}

We are now ready to prove our main result, namely:

\begin{theorem*}
Let $(\mathcal G,\mathbf w)$ and $(\mathcal G',\mathbf w')$ be two balanced geodesic graphs on the hyperbolic surface $(S,h)$. Then
$$\omega_\weil(t_{(\mathcal G,\mathbf w)},t_{(\mathcal G',\mathbf w')})=\frac{1}{2}\sum_{p\in e\cap e'}w_ew_e'\cos \theta_{e,e'}~,$$
where $e$ and $e'$ are intersecting edges of $\mathcal G$ and $\mathcal G'$ and $\theta_{e,e'}$ is the angle of intersection
between $e$ and $e'$  according to the orientation of $S$.

If $\mathcal G$ and $\mathcal G'$ have some non-trasverse intersection, the points of intersection must be counted by perturbing one of the two graphs by a small isotopy, so as to make the intersection transverse.
\end{theorem*}

\begin{proof}
Let $\alpha$ and $\alpha'$ be the closed 1-forms constructed in the previous subsection, so that 
$\Psi([\alpha]_\rham)=\Phi(\mathcal G,\mathbf w)$ and $\Psi([\alpha' ]_\rham)=\Phi(\mathcal G',\mathbf w')$. By Goldman theorem (recall Equation \eqref{eq goldman theorem}), the definition in Equation \eqref{eq defi variation teich} and Equation \eqref{eq goldman form in de rham}, we have:

$$\omega_\weil(t_{(\mathcal G,\mathbf w)},t_{(\mathcal G',\mathbf w')})=\frac{1}{4}\omega_{\mathrm G}(\Phi(\mathcal G,\mathbf w),\Phi(\mathcal G',\mathbf w'))=\frac{1}{4}\int_S \kappa(\alpha\wedge\alpha')~,$$

To compute the integral, first suppose that all points of intersection of $\mathcal G$ and $\mathcal G'$ are transverse and do not coincide with any of the vertices in $\mathcal V$ and $\mathcal V'$.  Replace the $\epsilon$ in the construction of $\alpha$ and $\alpha'$ by a smaller $\epsilon$ if necessary, so that $N_\epsilon(\mathcal G)$ and $N_\epsilon(\mathcal G')$ intersect only in the complement of the vertices of $\mathcal G$ and $\mathcal G'$. Since $\alpha$ vanishes in the complement of $N_\epsilon(\mathcal G)$, and $\alpha'$ vanishes in the complement of $N_\epsilon(\mathcal G')$, we have
$$\omega_\weil(t_{(\mathcal G,\mathbf w)},t_{(\mathcal G',\mathbf w')})=\frac{1}{4}\int_{N_\epsilon(\mathcal G)\cap N_\epsilon(\mathcal G')} \kappa(\alpha\wedge\alpha')~.$$
Now, for every point of intersection $p$ of edges $e$ of $\mathcal G$ and $e'$ of $\mathcal G'$, let $Q_p$ be the connected component of $N_\epsilon(\mathcal G)\cap N_\epsilon(\mathcal G')$ containing $p$. See Figure \ref{fig:components intersection}.


\begin{figure}
\begin{center}
\psfrag{p}{$p$}
\psfrag{N2}{$N_\epsilon(\mathcal{G}')$}
\psfrag{N}{$N_\epsilon(\mathcal G)$}
\psfrag{Q}{$Q_p$}
\includegraphics[scale=1]{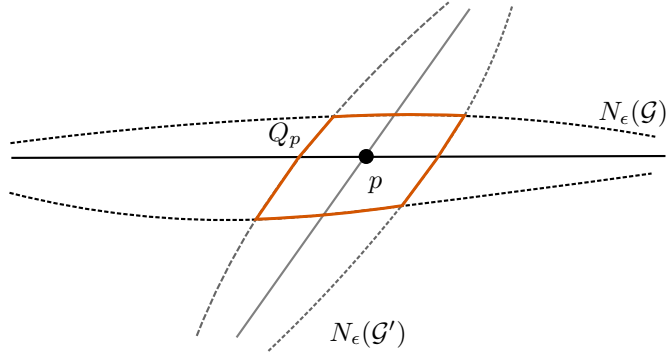}
\caption{The regions $Q_p$ in the intersection of the tubular neighborhoods of $\mathcal G$ and $\mathcal{G}'$, which are used in the proof.}
\label{fig:components intersection}
\end{center}
\end{figure}

Hence
$$\omega_\weil(t_{(\mathcal G,\mathbf w)},t_{(\mathcal G',\mathbf w')})=\frac{1}{4}\sum_{p\in e\cap e'}\int_{Q_p} \kappa(\alpha\wedge\alpha')~.$$
Let us now analyze each of the regions $Q_p$. Recall that $\alpha=\left(w_e {\mathfrak t}(\dev(\widetilde e))\right) d(f\circ \delta)$ inside $Q_p$, from the definition in Equation \eqref{eq defi derham form}. Analogously, let us write $\alpha=\left(w'_{e'} {\mathfrak t}(\dev(\widetilde e'))\right) d(f\circ \delta')$. Observe moreover that, if $x$ and $x'$ are the unit orthogonal vectors in $\R^{2,1}$ to the geodesics containing ${\mathfrak t}(\dev(\widetilde e))$ and ${\mathfrak t}(\dev(\widetilde e'))$ respectively (where their direction is induced by the orientation), then  by Equation \eqref{eq killing mink},
$$\kappa({\mathfrak t}(\dev(\widetilde e)),{\mathfrak t}(\dev(\widetilde e')))=2\langle x,x'\rangle=2\cos\theta_{e,e'}~.$$

Hence we have
$$\int_{Q_p}\kappa(\alpha\wedge\alpha')=2w_e w'_{e'}\cos\theta_{e,e'}\int_{Q_p}d(f\circ\delta)\wedge d(f\circ\delta')~.$$
We claim that $\int_{Q_p}d(f\circ\delta)\wedge d(f\circ\delta')=1$, which will thus conclude the proof that 
$$\int_{Q_p}\kappa(\alpha\wedge\alpha')=2w_e w'_{e'}\cos\theta_{e,e'}$$
and therefore 
$$\omega_\weil(t_{(\mathcal G,\mathbf w)},t_{(\mathcal G',\mathbf w')})=\frac{1}{2}\sum_{p\in e\cap e'}w_ew_{e'}\cos \theta_{e,e'}~.$$
To show the claim, by Stokes Theorem we have:
\begin{align*}
\int_{Q_p}d(f\circ\delta)\wedge d(f\circ\delta')=&
\int_{Q_p}d((f\circ\delta) d(f\circ\delta'))=\int_{\partial Q_p}(f\circ\delta) d(f\circ\delta')~.
\end{align*}
The region $Q_p$ has four smooth boundary components, namely:
$$(\partial Q_p)_{\pm}=\overline Q_p\cap \left\{f\circ \delta=\pm\frac{1}{2}\right\}$$
and
$$(\partial Q_p)'_{\pm}=\overline Q_p\cap \left\{f\circ \delta'=\pm\frac{1}{2}\right\}~.$$
Since $f'\circ \delta'$ is constant on $(\partial Q_p)'_{+}$ and $(\partial Q_p)'_{-}$, these two components do not contribute to the integral. On the other hand, let $q_2$ and $q_1$ be the two endpoints of $(\partial Q_p)_{+}$. By taking the orientation into account,
$$\int_{(\partial Q_p)_+}(f\circ\delta) d(f\circ\delta')=\frac{1}{2}\int_{(\partial Q_p)_+} d(f\circ\delta')=\frac{1}{2}(f\circ \delta'(q_2)-f\circ \delta'(q_1))=\frac{1}{2}~.$$
Similarly, we obtain 
$$\int_{(\partial Q_p)_-}(f\circ\delta) d(f\circ\delta')=\frac{1}{2}~,$$
since $f\circ \delta=-1/2$ on $(\partial Q_p)_-$ but the component $(\partial Q_p)_-$ is endowed with the opposite orientation from the orientation of $Q_p$, and therefore the two minus signs cancel out. This therefore shows $\int_{Q_p}d(f\circ\delta)\wedge d(f\circ\delta')=1$ as claimed, and thus concludes the proof under the hypothesis of transversality.

When the intersection of $\mathcal G$ and $\mathcal G'$ is not transverse, or contains some of the vertices (hence $\mathcal V\cap\mathcal V'\neq \emptyset$), then we can perturb $\mathcal G$ by an isotopy $h_t:S\to S$ (with $t\in[0,1]$ and $h_0=\mathrm{id}$), so as to make $h_1(\mathcal G)$ transverse to $\mathcal G'$ and to ensure that there are no vertices in common between $h_1(\mathcal G)$ and $\mathcal G'$, see Figures \ref{fig:perturbation1} and \ref{fig:perturbation2}. Moreover, we can also assume that the intersection of $h_1(\mathcal G)$ with an $\epsilon$-neighborhood of $\mathcal G'$ is geodesic. Since $\alpha$ and $h_1^*\alpha$ are in the same de Rham cohomology class, we can repeat the same computation using $h_1^*\alpha$ and $\alpha'$. Hence the result is the same as before, namely
$$\omega_\weil(t_{(\mathcal G,\mathbf w)},t_{(\mathcal G',\mathbf w')})=\frac{1}{2}\sum_{p\in h_1(e)\cap e'}w_ew_e'\cos \theta_{e,e'}~,$$
where $p$ are the intersection points of $h_1(\mathcal G)$ and $\mathcal G'$, belonging to edges $h_1(e)$ and $e'$, and the angle $\theta_{e,e'}$ is the angle between the geodesic edges $e$ and $e'$ of the original balanced geodesic graphs.
\end{proof}


\begin{figure}
\begin{center}
\includegraphics[scale=1]{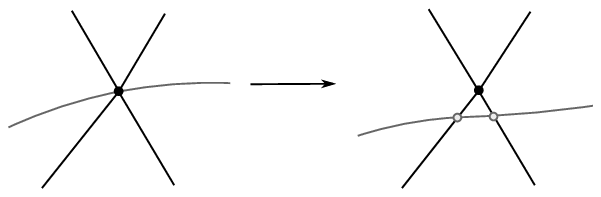}
\caption{Perturbation by a small isotopy, when $\mathcal G$ and $\mathcal{G}'$ share a common vertex.
\label{fig:perturbation1}}
\end{center}
\end{figure}

\begin{figure}
\begin{center}
\includegraphics[scale=1]{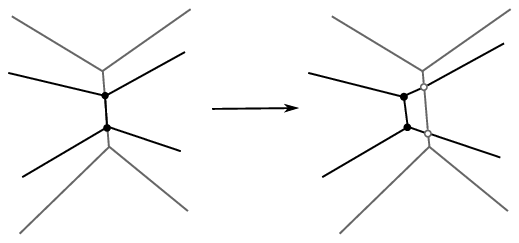}
\caption{Perturbation by a small isotopy, when $\mathcal G$ and $\mathcal G'$ share an edge.}
\label{fig:perturbation2}
\end{center}
\end{figure}


As we said, if $\mathcal G$ and $\mathcal G'$ are not transverse, or intersect on vertices, the only caveat in applying the formula \eqref{eq main formula} is to slightly perturb (topologically) one of the two graphs and use the perturbed graph to count intersection points. 
It follows from the Theorem that the result will not change if one chooses two different isotopies to perturb $\mathcal G$, which therefore result in different edge intersections. 
Let us remark that this is also a consequence of the balance condition \eqref{eq:mink sum}.

\bibliographystyle{alpha}
\bibliographystyle{ieeetr}
\bibliography{bs-bibliography}

\end{document}